\documentclass[11pt]{amsart}

\usepackage[margin=1.15in]{geometry}
\usepackage{amssymb,mathtools}
\newcommand{\setZ}{\mathbb Z}
\providecommand\given{}
\newcommand\SetSymbol[1][]{%
  \nonscript\:#1\vert\allowbreak\nonscript\:\mathopen{}%
}
\DeclarePairedDelimiterX\Set[1]\{\}{%
  \renewcommand\given{\SetSymbol[\delimsize]}#1%
}
\DeclareMathOperator{\rk}{rk}
\DeclareMathOperator{\Mat}{M}
\DeclareMathOperator{\diag}{diag}
\DeclareMathOperator{\id}{id}

\DeclareMathOperator{\GL}{GL}
\DeclareMathOperator{\Rea}{Re}
\newcommand{\E}{\operatorname{E}}
\usepackage{enumitem}
\usepackage{needspace}
\usepackage{xcolor}
\usepackage{tikz}
\usetikzlibrary{arrows.meta,calc}
\tikzset{
  bratteli edge/.style={-{Stealth[length=1.6mm,width=1.1mm]},
  line width=.45pt,shorten >=1.7pt,shorten <=1.7pt}
}

\newtheorem{theorem}{Theorem}[section]
\newtheorem{proposition}[theorem]{Proposition}
\newtheorem{lemma}[theorem]{Lemma}
\newtheorem{corollary}[theorem]{Corollary}
\theoremstyle{definition}
\newtheorem{definition}[theorem]{Definition}

\theoremstyle{remark}
\newtheorem{remark}[theorem]{Remark}

\usepackage[colorlinks=true, hypertexnames=false,
  linkcolor=blue!55!black, citecolor=blue!55!black,
  urlcolor=blue!55!black]{hyperref}
\usepackage[nameinlink, noabbrev, capitalise]{cleveref}
\theoremstyle{plain}
\newtheorem*{repeatedmaintheorem}{Theorem~\ref{thm:main}}
\theoremstyle{remark}
\newcommand{\MainTheoremStatement}{%
The group \(\GL(Q)\) is exotic, and
\[
 \overline{\E_\mu(R)}=\overline{\GL_\mu(R)}=\GL(Q).
\]
}
\AddToHook{env/theorem/begin}{\crefalias{section}{theorem}}
\AddToHook{env/proposition/begin}{\crefalias{section}{proposition}}
\AddToHook{env/lemma/begin}{\crefalias{section}{lemma}}
\AddToHook{env/corollary/begin}{\crefalias{section}{corollary}}
\AddToHook{env/definition/begin}{\crefalias{section}{definition}}
\AddToHook{env/example/begin}{\crefalias{section}{example}}
\AddToHook{env/remark/begin}{\crefalias{section}{remark}}

\title{Exotic Unit Groups of Sylvester Rank Completions}

\author{Baojie Jiang}
\address{\hskip-\parindent
  Baojie Jiang,
  School of Mathematical Sciences, Chongqing Normal University, University Town, Shapingba District, Chongqing, 401331, China.}
\email{jiangbaojie@gmail.com}

\date{\today}
\subjclass[2010]{Primary 22A25; Secondary 16E50}
\keywords{Sylvester matrix rank function, rank-metric completion,
exotic group, unitary representation, escape function,
Bourbaki boundedness, extreme amenability}

\hypersetup{
  pdftitle={Exotic Unit Groups of Sylvester Rank Completions},
  pdfauthor={Baojie Jiang},
  pdfsubject={Unitary representations and unit groups of Sylvester rank completions},
  pdfkeywords={Sylvester matrix rank function, rank-metric completion,
    exotic group, unitary representation, escape function,
Bourbaki boundedness, extreme amenability}
}

\begin{document}

\allowdisplaybreaks
\raggedbottom

\begin{abstract}
Let \(R\) be any unital ring equipped with a Sylvester matrix rank
function, and let \(Q\) be the rank completion of the
matrix system defined by block-diagonal embeddings along a factor sequence.
We prove that both \(\GL(Q)\) and \((Q,+)\) are exotic:
every strongly continuous unitary representation is trivial.
We also prove that \(\GL(Q)\) is Bourbaki bounded and admits no
nonzero escape function.
For every \(\varepsilon>0\), let \(B_\varepsilon\) be the open ball
centered at the identity. The least positive integer \(N\) such that
\(\GL(Q)=B_\varepsilon^N\) is
\(\lfloor\varepsilon^{-1}\rfloor+1\).
Neither regularity nor irreducibility is required.
Our proof of exoticness uses a uniform Kazhdan estimate for additive integer
matrix groups together with rank-one matrix perturbations, rather than the
character-theoretic and continuous ring methods used previously.
The geometric conclusions follow from factorizations using
diagonal corner subgroups.
\end{abstract}

\maketitle
\enlargethispage{2pt}

\section{Introduction}

Let \(D\) be a division ring and let \(\mu=(n_k)_{k\geq1}\)
be a \emph{factor sequence}, that is, an unbounded sequence of
positive integers such that \(n_k\mid n_{k+1}\).
We consider the block-diagonal embeddings
\[
 \phi_k:\Mat_{n_k}(D)\longrightarrow\Mat_{n_{k+1}}(D),
 \qquad
 \phi_k(X)=\diag(\underbrace{X,\ldots,X}_{n_{k+1}/n_k\text{ copies}}).
\]
These maps preserve the normalized rank \(\frac{1}{n_k}\rk(X)\), where
\(\rk\) is the usual matrix rank over \(D\).
Hence the metrics
\(d(X,Y)=\frac{1}{n_k}\rk(X-Y)\) induce a metric on the direct limit
\(D_\mu\) of this system.
Its completion, denoted by \(\mathcal M_D\), is a continuous factor.

This construction goes back to von Neumann's theory of continuous
geometry \cite{vonNeumann1936Geometry}.
Independence of the chosen factor sequence was established by
von Neumann and developed by Halperin \cite{Halperin1968};
further uniqueness and approximation results were obtained by
Ara and Claramunt \cite{AraClaramunt2018}.
The rank topology on the unit groups of these completions naturally leads to
questions in unitary representation theory and topological dynamics.

A topological group is called \emph{exotic} if every strongly
continuous unitary representation on a complex Hilbert space
is trivial.
Such groups were first constructed by Herer and
Christensen \cite{HererChristensen} using pathological submeasures.
Carderi and Thom \cite{CarderiThom} exhibited an exotic matrix group:
\(\GL(\mathcal M_{\mathbb F_q})\) is exotic and extremely
amenable. Their proofs use character estimates for finite special
linear groups and concentration of measure, respectively.

Recently, Schneider and Thom \cite[Theorem~1.1]{SchneiderThom}
proved exoticness of the unit group of every non-discrete
irreducible continuous ring.
This includes the rings of operators
affiliated with arbitrary \(\mathrm{II}_1\) factors.
Their argument uses character estimates in positive characteristic,
charmenable arithmetic groups in characteristic zero, and character
rigidity together with the structure of continuous rings.
Schneider \cite{SchneiderGeometry}
also established Bourbaki boundedness and absence of nonzero
escape functions for these unit groups, using small-subgroup
factorizations derived from Schneider's continuous triangularization theorem
\cite[Theorem~9.11]{SchneiderTriangularization}.
Related structural methods yield width bounds and automatic continuity in the
work of Bernard and Schneider \cite{BernardSchneider}.

Sylvester matrix ranks allow the same block-diagonal construction
over more general rings.
The completion can then be
non-regular and non-simple, as happens for the normalized length
rank on \(\setZ/4\setZ\)
\cite[Example~5.6 and Remark~5.7]{JiangBratteli}.
Such examples lie outside the scope of the continuous ring results above.
We show that the rank axioms and block-diagonal embeddings alone suffice to
establish exoticness, Bourbaki boundedness,
and absence of nonzero escape functions.
Conversely, the continuous ring theorems apply to rings that need not
arise from these matrix systems.

We now let \(R\) be a unital ring equipped with a Sylvester matrix rank function
\(\mathrm{rk}\), and call \((R,\mathrm{rk})\) a \emph{rank ring}.
Using the same block-diagonal construction along the factor sequence \(\mu\), let
\(R_\mu=\varinjlim_k\Mat_{n_k}(R)\) and write
\(Q=\overline{R_\mu}\) for its completion in the normalized
rank metric, after identifying elements at distance zero.
The construction is detailed in \Cref{sec:factor-sequences};
neither regularity of \(R\) nor faithfulness of \(\rk\) is assumed.

Write \(\GL(Q)\) for the unit group of \(Q\), and let \(\iota_k:\Mat_{n_k}(R)\to Q\)
be the canonical maps.
For \(n\geq 1\), write \(\GL_n(R)\) for the group of invertible \(n\times n\)
matrices over \(R\), and \(\E_n(R)\) for its elementary subgroup.
Let \(\E_\mu(R)\) and \(\GL_\mu(R)\) be the unions of the
images in \(Q\) of \(\E_{n_k}(R)\) and \(\GL_{n_k}(R)\),
respectively.
Subgroups of \(\GL(Q)\) carry the inherited rank topology, and closures are taken in \(Q\).

The key point is that elementary matrices already determine every continuous
unitary representation of \(\GL(Q)\): they generate a dense subgroup on which
every such representation is trivial.
More precisely, we prove the following.

%

\begin{repeatedmaintheorem}
\MainTheoremStatement
\end{repeatedmaintheorem}

The representation-theoretic part of the proof rests on a Kazhdan estimate for
additive integer matrix groups whose constant is independent of the matrix size.
After amplifying an elementary matrix \(u=1_{n_k}+ae_{ij}\) to level \(\ell>k\),
we realize its action through a representation of the additive group
\(M_r(\mathbb Z)\), where \(r=n_\ell/n_k\). The matrices \(vw^T\) occurring in
this estimate give perturbations of the identity of normalized rank at most
\(1/n_\ell\). The uniform Kazhdan estimate, together with strong continuity,
then forces every unitary representation of \(\GL(Q)\) to be trivial on
\(\E_\mu(R)\). The estimate is obtained from the codistance criterion of Ershov
and Jaikin-Zapirain \cite[Lemma~3.1(a)]{EJ}; see
\Cref{cor:integer-matrix-displacement,thm:roots}.

To pass from elementary matrices to the full group \(\mathrm{GL}(Q)\), we prove the density of
\(\E_\mu(R)\) in \(\GL(Q)\) using a quantitative version of the Whitehead block
identity \cite[Example~III.1.2.1]{WeibelKbook}. The same integer-matrix estimate
also yields exoticness of the additive group \((Q,+)\); see \Cref{thm:additive}.

A different consequence of the block-diagonal structure is a sharp factorization
into small rank balls. For \(\varepsilon>0\), let \(B_\varepsilon\) be the open
ball centered at \(1_Q\), and let \(B_\varepsilon^N\) denote the set of products
of \(N\) elements of this ball.
We prove that
\[
 \min\Set*{N\in\mathbb N_{\geq1}\given\GL(Q)=B_\varepsilon^N}
 =\lfloor\varepsilon^{-1}\rfloor+1.
\]
In particular, \(\GL(Q)\) is \emph{Bourbaki bounded}: for each identity
neighborhood \(U\), finitely many left translates of some finite power of \(U\)
cover the group. We also prove that every escape function on \(\GL(Q)\) is zero,
and that every continuous homomorphism from \(\GL(Q)\) to a Hausdorff group with
the escape property is trivial; see
\Cref{cor:bourbaki,cor:no-escape,cor:escape-homomorphisms}. Here an \emph{escape
  function} is a length function that is uniformly small on elements whose first
sufficiently many powers remain in a fixed identity neighborhood. A Hausdorff
group has the \emph{escape property} if every identity neighborhood contains
\(f^{-1}([0,1))\) for some escape function \(f\). Precise definitions are given
in \Cref{sec:small-subgroups}.

Combining exoticness with amenability, we obtain extreme amenability of
\((Q,+)\) and, when \(R\) is locally finite, of \(\GL(Q)\); see
\Cref{thm:extreme-amenability,thm:unit-extreme-amenability}. Our results do not
settle the amenability question for \(\GL(\mathcal M_{\mathbb Q})\) raised in
\cite[Remark~5.12(1)]{SchneiderAmenability} and
\cite[Problem~6.1]{SchneiderThom}. The rank-preserving uniqueness theorem of
\cite[Theorem~1.1]{JiangBratteli} also transfers these results to weighted
Bratteli completions under the hypotheses recorded in \Cref{sec:bratteli}.

The paper is organized as follows. \Cref{sec:setting} gives the preliminaries.
\Cref{sec:integer-estimates} proves the integer-matrix estimate, and
\Cref{sec:density} applies it to elementary matrices and establishes density,
yielding unit-group exoticness. \Cref{sec:additive} proves additive-group
exoticness, while \Cref{sec:small-subgroups} establishes the geometric
conclusions. Finally, \Cref{sec:applications} contains the fixed-point results,
non-regular examples, and the Bratteli application.

\section*{Acknowledgments}

\paragraph{\textbf{AI Use Statement.}}
This work was completed with the assistance of ChatGPT
(GPT-5.6 Sol and GPT-6 Astra) in developing and checking
the mathematical arguments.
The author takes full responsibility for all mathematical
statements, proofs, references, and editorial decisions
in the final version.

\section{Preliminaries}\label{sec:setting}

For \(t\in\mathbb R\), let \(\lfloor t\rfloor\) denote the greatest
integer not exceeding \(t\), and let \(\lceil t\rceil\) denote the
least integer not less than \(t\).

The identity of a unital ring \(R\) is denoted by \(1_R\).
The ring is \emph{von Neumann regular} if every \(a\in R\) has
an element \(b\in R\) with \(a=aba\).
An element \(e\in R\) is an \emph{idempotent} if \(e^2=e\).
The \emph{corner ring} \(eRe=\Set*{exe\given x\in R}\)
has identity \(e\).

For positive integers \(m,n\), write \(\Mat_{m\times n}(R)\) for the set of
\(m\times n\) matrices over \(R\), and put \(\Mat_n(R)=\Mat_{n\times n}(R)\).
Its identity matrix is \(1_n\), and
\(e_{ij}\) denotes the matrix of the appropriate size whose \((i,j)\)-entry is
\(1_R\) and whose other entries are zero.
We write \(0_n\) for the zero \(n\)-by-\(n\) matrix and \(e_i\)
for the \(i\)-th standard column; the coefficient ring and size
are determined by the context.
For any \(A=(a_{ij})\in \Mat_{m\times n}(R)\), its transpose is
\(A^{\mathsf T}=(a_{ji})\).
For a ring homomorphism \(\psi:S\to T\) and a rectangular matrix
\(A=(a_{ij})\) over \(S\), we write \(\psi(A)=(\psi(a_{ij}))\)
for the matrix obtained by applying \(\psi\) to each entry.
For an abelian group \(V\), the notation \(V^n\) denotes its
\(n\)-fold direct product; when rows and columns must be distinguished,
we write \(V^{1\times n}\) and \(V^{n\times1}\), respectively.

For \(A=(a_{ij})\in\Mat_{r\times s}(R)\) and
\(B\in\Mat_{m\times n}(R)\), we use the Kronecker product
\[
 A\otimes B=(a_{ij}B)_{i,j}\in\Mat_{rm\times sn}(R),
\]
with lexicographically ordered row and column index pairs.
In particular,
\[
 1_r\otimes X=\diag(\underbrace{X,\ldots,X}_{r\text{ copies}}).
\]
Integer coefficients act through the canonical homomorphism
\(\setZ\to R\), whose image lies in the centre of \(R\).

\Needspace{12\baselineskip}
\subsection{Sylvester matrix rank functions}

We use the standard axioms for Sylvester matrix rank functions;
see Malcolmson \cite{Malcolmson1980} for their relation to
homomorphisms to division rings. For further background, see
\cite{JaikinZapirain2019,JiangLi2021,Li2021} and
\cite[Part~I, Chapter~7]{Schofield1985}.

\begin{definition}
Let \(R\) be a unital ring. A \emph{Sylvester matrix rank function}
on \(R\) is an \(\mathbb R_{\geq0}\)-valued function \(\rk\) on
all finite rectangular matrices over \(R\), satisfying:
\begin{enumerate}[label=(SM\arabic*)]
\item \(\rk(0)=0\) and \(\rk(1_R)=1\);
\item \(\rk(AB)\leq\min\Set*{\rk(A),\rk(B)}\) for compatible matrices;
\item \(\rk\diag(A,B)=\rk(A)+\rk(B)\);
\item \(\rk\begin{pmatrix}A&C\\0&B\end{pmatrix}\geq\rk(A)+\rk(B)\)
for matrices of the indicated sizes.
\end{enumerate}
\end{definition}

A \emph{rank ring} \((R,\rk)\) is a unital ring with a specified
Sylvester matrix rank function. The rank is \emph{faithful} if
\(\rk(A)>0\) for every nonzero matrix \(A\).

The rank axioms imply subadditivity,
invariance under invertible row and column operations, and
\(\rk(A)\leq\min\Set*{m,n}\) for \(A\in\Mat_{m\times n}(R)\).
Indeed, \(\rk(1_m)=m\); writing \(A=1_mA1_n\) gives the size bound.
Applying the product inequality first to \(SAT\) and then to
\(A=S^{-1}(SAT)T^{-1}\) proves invariance when \(S,T\) are
invertible. For \(A,B\in\Mat_{m\times n}(R)\), the factorization
\[
 A+B=\begin{pmatrix}1_m&1_m\end{pmatrix}
       \diag(A,B)\begin{pmatrix}1_n\\1_n\end{pmatrix}
\]
proves subadditivity.
In particular
\( |\rk(A)-\rk(B)|\leq\rk(A-B)\) for equal-size matrices.
A product of a column and a row has rank at most \(1\),
regardless of the sizes or values of its entries.

The \emph{rank pseudometric} on \(R\) is
\[
 d_{\rk}(x,y):=\rk(x-y)\qquad(x,y\in R).
\]
Symmetry and the triangle inequality follow from invariance under
multiplication by \(-1_R\) and subadditivity, respectively.
If \(\rk\) is faithful, this is a metric, called the \emph{rank metric}.
In either case, the induced topology is the \emph{rank topology}.

For \(X\in\Mat_{r\times s}(\Mat_p(R))\), under the identification
\(\Mat_{r\times s}(\Mat_p(R))\cong\Mat_{rp\times sp}(R)\), set
\[
 \rk_p(X):=\frac1p\rk(X).
\]
This defines a Sylvester matrix rank function on \(\Mat_p(R)\),
whose rank pseudometric is \((X,Y)\mapsto\rk_p(X-Y)\).

\subsection{Factor sequences and rank completions}\label{sec:factor-sequences}

Fix a rank ring \((R,\rk)\). A \emph{factor sequence} is an
unbounded sequence \(\mu=(n_k)_{k\geq1}\) of positive integers
such that \(n_k\mid n_{k+1}\) for every \(k\geq1\).
Fix such a sequence and put \(n_0=1\). The block-diagonal embeddings
\[
 \phi_{\ell,k}:\Mat_{n_k}(R)\longrightarrow\Mat_{n_\ell}(R),\qquad
 \phi_{\ell,k}(X)=1_{n_\ell/n_k}\otimes X
 \quad(\ell\geq k\geq0)
\]
form a direct system. For \(k\geq1\), the consecutive map
\(\phi_{k+1,k}\) is the embedding \(\phi_k\) used in the
introduction. These maps preserve the normalized ranks:
\[
 \rk_{n_\ell}(\phi_{\ell,k}(X))=\rk_{n_k}(X).
\]
Thus the algebraic direct limit
\[
 R_\mu:=\varinjlim_{k\geq0}(\Mat_{n_k}(R),\phi_{k+1,k})
\]
carries the induced Sylvester matrix rank function \(\rk_\mu\).

The set \(\ker(\rk_\mu)=\Set*{x\in R_\mu\given\rk_\mu(x)=0}\)
is a two-sided ideal, and the rank descends to the quotient.
We write
\[
 Q=\overline{R_\mu}
\]
for the metric completion of \(R_\mu/\ker(\rk_\mu)\).
The rank axioms ensure that the ring operations and the Sylvester
matrix rank extend uniquely and continuously to \(Q\).
We denote the extended rank by \(\rk_\mu\) and write
\(d(x,y)=\rk_\mu(x-y)\) for the metric on \(Q\).

Let \(\iota_k:\Mat_{n_k}(R)\to Q\) be the canonical unital
homomorphism. Then
\begin{equation}\label{eq:canonical-compatibility}
 \iota_\ell\circ\phi_{\ell,k}=\iota_k
 \qquad(\ell\geq k),
\end{equation}
and, for \(X,Y\in\Mat_{n_k}(R)\),
\begin{equation}\label{eq:canonical-rank}
 \rk_\mu(\iota_k(X))=\rk_{n_k}(X),\qquad
 d(\iota_k(X),\iota_k(Y))=\rk_{n_k}(X-Y).
\end{equation}
The images \(\iota_k(\Mat_{n_k}(R))\) form an increasing union
dense in \(Q\). Since the rank on \(R\) need not be faithful,
we distinguish a matrix \(X\) from its image \(\iota_k(X)\).

We equip the unit group \(\GL(Q)\) with the inherited rank
metric and denote its identity by \(1_Q\).
Multiplication by a unit preserves rank, and
\[
 d(u^{-1},v^{-1})=d(u,v)\qquad(u,v\in\GL(Q)).
\]
Thus \(\GL(Q)\) is a topological group with a bi-invariant metric.
Its closedness in \(Q\) is proved in \Cref{prop:density}.
For \(n\geq1\), write \(\GL_n(R)=\GL(\Mat_n(R))\).
For \(n\geq2\), let \(\E_n(R)\) be the subgroup generated by
\(1_n+ae_{ij}\), where \(a\in R\) and \(i\ne j\), and put
\(\E_1(R)=\Set{1_R}\). Define
\[
 \E_\mu(R):=\bigcup_{k\geq0}\iota_k\bigl(\E_{n_k}(R)\bigr),
 \qquad
 \GL_\mu(R):=\bigcup_{k\geq0}\iota_k\bigl(\GL_{n_k}(R)\bigr).
\]
The block-diagonal embeddings preserve invertibility and send each
elementary generator to a product of elementary matrices. These are therefore
increasing unions of subgroups, with
\[
 \E_\mu(R)\leq\GL_\mu(R)\leq\GL(Q).
\]

The additive group of \(Q\) and the three groups above are
non-discrete. For \(n_k\geq2\), the matrix \(e_{12}\) has rank
\(1\): it factors through a single coordinate, and extracting
its \((1,2)\) entry gives the reverse inequality. Hence
\[
 d(\iota_k(e_{12}),0)
 =d(\iota_k(1_{n_k}+e_{12}),1_Q)=\frac1{n_k}>0,
\]
and these distances tend to zero.
If \(R\) is countable, then \(Q\) is separable; completeness of
\(Q\) and closedness of \(\GL(Q)\) imply that both \((Q,+)\)
and \(\GL(Q)\) are Polish.

\subsection{Unitary representations}\label{sec:unitary-preliminaries}

All Hilbert spaces are complex, with inner products linear in the
first variable. For a Hilbert space \(\mathcal H\), write
\(B(\mathcal H)\) for its bounded linear
operators, \(U(\mathcal H)\) for its unitary
group, and \(\id_{\mathcal H}\) for its identity operator.
Put \(\mathbb T=\mathbb R/\setZ\), and write \(\mathrm{i}\) for the imaginary unit.

\begin{definition}
A \emph{representation} of a topological group \(G\) on
\(\mathcal H\) is a homomorphism \(\pi:G\to U(\mathcal H)\)
that is \emph{strongly continuous}: the map \(g\mapsto\pi(g)\xi\) is
continuous for each \(\xi\in\mathcal H\).
For a discrete group this continuity condition is automatic.
In particular,
\(\pi(gh)=\pi(g)\pi(h)\) and
\(\pi(g^{-1})=\pi(g)^{-1}=\pi(g)^*\).
The representation is \emph{trivial} if
\(\pi(g)=\id_{\mathcal H}\) for every \(g\in G\).
The group \(G\) is \emph{exotic} if every representation
of \(G\) is trivial.
\end{definition}

For a representation \(\sigma:G\to U(\mathcal H)\) and a subgroup
\(H\leq G\), write \(P_{H}\) for the orthogonal projection onto the closed
subspace of invariant vectors
\[
 \mathcal H^H
 :=\Set*{\xi\in\mathcal H\given \sigma(h)\xi=\xi\text{ for every }h\in H}.
\]
We write \(\operatorname{Rep}_0(G)\) for the class of representations
of \(G\) without nonzero invariant vectors.

We shall use the spectral theorem for representations of the discrete
group \(\setZ^d\). Identify \(\widehat{\setZ^d}\) with
\(\mathbb T^d\) through the pairing
\[
 \chi_n(\theta)=\langle n,\theta\rangle
 =\exp\left(2\pi\mathrm{i}\sum_{j=1}^d n_j\theta_j\right),
\]
where \(n=(n_j)_{j=1}^d\in\setZ^d\) and
\(\theta=(\theta_j)_{j=1}^d\in\mathbb T^d\); thus \(\theta\)
corresponds to the character \(n\mapsto\chi_n(\theta)\).
For a representation \(\sigma:\setZ^d\to U(\mathcal H)\), the spectral theorem
\cite[Theorem~D.3.1]{BekkaHarpeValette} gives a projection-valued
measure \(E\) on \(\mathbb T^d\). We use the convention
\[
 \sigma(n)=\int_{\mathbb T^d}\overline{\chi_n(\theta)}\,dE(\theta).
\]
For \(\xi\in\mathcal H\) and a Borel set \(B\subseteq\mathbb T^d\), put
\(\nu_\xi(B)=\langle E(B)\xi,\xi\rangle\). Then
\[
 \begin{aligned}
 \langle\sigma(n)\xi,\xi\rangle
 &=\int_{\mathbb T^d}\overline{\chi_n(\theta)}\,d\nu_\xi(\theta),\\
 \|\sigma(n)\xi\|^2
 &=\int_{\mathbb T^d}|\chi_n(\theta)|^2\,d\nu_\xi(\theta)
 =\nu_\xi(\mathbb T^d)=\|\xi\|^2.
 \end{aligned}
\]
In particular,
\[
 \|(\sigma(n)-\id_{\mathcal H})\xi\|^2
 =\int_{\mathbb T^d}|\chi_n(\theta)-1|^2\,d\nu_\xi(\theta),
\]
which identifies the fixed space of \(\sigma(n)\) with the range of
\(E(\Set*{\theta\in\mathbb T^d\given\chi_n(\theta)=1})\).

For \(H\leq\setZ^d\), its \emph{annihilator} is
\[
 \begin{aligned}
 H^\perp
 &=\Set*{\theta\in\mathbb T^d\given
                    \chi_n(\theta)=1\text{ for all }n\in H}\\
 &=\bigcap_{n\in H}\Set*{\theta\in\mathbb T^d\given\chi_n(\theta)=1}.
 \end{aligned}
\]
Since \(H\) is countable, intersecting these fixed spaces gives
\(P_H=E(H^\perp)\), and
\begin{equation}\label{eq:spectral-fixed-space}
 \|P_H\xi\|^2
 =\langle E(H^\perp)\xi,\xi\rangle
 =\nu_\xi(H^\perp).
\end{equation}
In particular, \(\nu_\xi(\{0\})=\|P_{\setZ^d}\xi\|^2\), so a
representation without nonzero invariant vectors has no spectral
mass at the origin.

\section{Kazhdan estimates for integer matrix groups}\label{sec:integer-estimates}

Throughout this section, all groups carry the discrete topology
and all representations are unitary.
For each \(r\geq1\), we prove that the matrices \(vw^{\mathsf T}\),
with \(v\in\{0,1\}^r\) and \(w\in\setZ^r\), form a Kazhdan set
in \((\Mat_r(\setZ),+)\) with Kazhdan constant at least \(1\).

Let \(G\) be a discrete group, let \(S\subseteq G\) be nonempty, and let \(\varepsilon>0\).
A nonzero vector \(\xi\) in a unitary representation \(\pi:G\to U(\mathcal H)\)
is \emph{\((S,\varepsilon)\)-invariant} if
\[
 \sup_{s\in S}\|\pi(s)\xi-\xi\|<\varepsilon\|\xi\|.
\]
The pair \((S,\varepsilon)\) is called a \emph{Kazhdan pair} if every
unitary representation admitting a \((S,\varepsilon)\)-invariant
vector has a nonzero \(G\)-invariant vector.
A subset \(S\) is a \emph{Kazhdan set} if \((S,\varepsilon)\) is a Kazhdan pair
for some \(\varepsilon>0\).
We refer to \cite[Section~1.1]{BekkaHarpeValette} for background on
Kazhdan sets and Property~\textup{(T)}.
For a nonempty subset \(S\subseteq G\), the \emph{Kazhdan constant} is
\[
 \kappa(G,S)=\inf_{\substack{\pi:G\to U(\mathcal H)\\
             \pi\in\operatorname{Rep}_0(G),\ 0\neq\xi\in\mathcal H}}
 \frac{\sup_{s\in S}\|\pi(s)\xi-\xi\|}{\|\xi\|}.
\]
We use the convention \(\inf\varnothing=+\infty\).
Thus \((S,\varepsilon)\) is a Kazhdan pair precisely when
\(\kappa(G,S)\geq\varepsilon\).
Kazhdan sets need not be finite. For examples in infinite discrete
abelian groups, which do not have Property~\textup{(T)}, see
\cite{BadeaGrivaux}.

A Kazhdan pair gives a bound on the action of the whole group
in terms of its action on the given subset.

\begin{lemma}\label{lem:kazhdan-separation}
Let \(G\) be a discrete group and let \((S,\varepsilon)\) be a Kazhdan pair.
Then, for every unitary representation
\(\pi:G\to U(\mathcal H)\) and any \(\xi\in\mathcal H\),
\begin{equation}\label{eq:kazhdan-separation}
 \sup_{g\in G}\|\pi(g)\xi-\xi\|
 \leq
 \frac{2}{\varepsilon}
 \sup_{s\in S}\|\pi(s)\xi-\xi\|.
\end{equation}
\end{lemma}

\begin{proof}
Put \(\eta=(\id_{\mathcal H}-P_G)\xi\).
If \(\eta=0\), the assertion is immediate. Otherwise the restriction of \(\pi\) to
\((\mathcal H^G)^\perp\) has no nonzero invariant vectors, so the
Kazhdan assumption gives
\[
 \varepsilon\|\eta\|
 \leq
 \sup_{s\in S}\|\pi(s)\eta-\eta\|
 =
 \sup_{s\in S}\|\pi(s)\xi-\xi\|.
\]
On the other hand, for every \(g\in G\),
\[
 \|\pi(g)\xi-\xi\|
 =
 \|\pi(g)\eta-\eta\|
 \leq2\|\eta\|.
\]
Combining the two inequalities proves the claim.
\end{proof}

To obtain a Kazhdan constant independent of the matrix size, we
use the codistance criterion of Ershov and Jaikin-Zapirain.
Following \cite[Section~2]{EJ}, define the \emph{codistance}
of closed subspaces \(U_1,\ldots,U_n\) of a Hilbert space
\(\mathcal H\) by
\[
 \rho(U_1,\ldots,U_n)
 =
 \sup_{\substack{u_i\in U_i\\(u_1,\ldots,u_n)\neq0}}
 \frac{\|u_1+\cdots+u_n\|^2}
 {n\sum_{i=1}^n\|u_i\|^2}.
\]
We set this value to zero if all \(U_i\) are zero.
If \(P_i\) denotes the orthogonal projection onto \(U_i\), then
\begin{equation}\label{eq:codistance-projection}
 \rho(U_1,\ldots,U_n)
 =
 \frac1n
 \sup_{\xi\neq0}
 \frac{\sum_{i=1}^n\|P_i\xi\|^2}{\|\xi\|^2}.
\end{equation}
This follows from \(\|T\|=\|T^*\|\) for the summation operator
\(T:\bigoplus_{i=1}^n U_i\to\mathcal H\), whose adjoint is
\(T^*\xi=(P_1\xi,\ldots,P_n\xi)\).
When \(\mathcal H=\{0\}\), both sides are understood to be zero.

For subgroups \(H_1,\ldots,H_n\leq G\), define
\[
 \rho(\{H_i\}_{i=1}^n,G)
 =
 \sup_{\substack{\pi:G\to U(\mathcal H)\\
                 \pi\in\operatorname{Rep}_0(G)}}
 \rho\bigl(
   \mathcal H^{H_1},\ldots,\mathcal H^{H_n}
 \bigr).
\]
If \(H_1,\ldots,H_n\) generate \(G\) and
\(\rho(\{H_i\}_{i=1}^n,G)<1\), then
\cite[Lemma~3.1(a)]{EJ} gives
\begin{equation}\label{eq:codistance-kazhdan}
 \kappa\left(
   G,\bigcup_{i=1}^nH_i
 \right)
 \geq
 \sqrt{2\bigl(1-\rho(\{H_i\}_{i=1}^n,G)\bigr)}.
\end{equation}

For the integer matrix group, we use subgroups consisting of rank-one matrices
whose column factor lies in the finite cube \(\{0,1\}^r\).

\begin{proposition}\label{prop:integer-matrices-kazhdan}
  Let \(r\geq1\), let \(G=(\Mat_r(\setZ),+)\) and let \(V=\{0,1\}^r\).
  For \(v\in V\), set \(H_v = \Set*{
   vw^{\mathsf T}
   \given
   w\in\setZ^r
 }\).
Then
\[
 \rho(\{H_v\}_{v\in V},G)\leq\frac12.
\]
Consequently, if
\[
 S=\bigcup_{v\in V}H_v
 =
 \Set*{
   vw^{\mathsf T}
   \given
   v\in\{0,1\}^r,\;
   w\in\setZ^r
 },
\]
then \((S,1)\) is a Kazhdan pair for \(G\).
\end{proposition}

\begin{proof}
  For \(1\leq i,j\leq r\), we have
  \[
    e_{ij}=e_i e_j^{\mathsf T}\in H_{e_i}.
  \]
Thus the subgroups \(H_v\), \(v\in V\), generate \(G\).

Let \(\pi:G\to U(\mathcal H)\) belong to \(\operatorname{Rep}_0(G)\).
For \(v\in V\), write \(P_v=P_{H_v}\) for the orthogonal
projection onto \(\mathcal H^{H_v}\). We shall prove that
\begin{equation}\label{eq:matrix-projection-average}
 \frac1{2^r}\sum_{v\in V}
 \|P_v\xi\|^2
 \leq
 \frac12\|\xi\|^2
 \qquad(\xi\in\mathcal H).
\end{equation}

Write \(\Mat_r(\mathbb T)\) for the additive group of
\(r\times r\) matrices with entries in \(\mathbb T\).
Identify the Pontryagin dual of \(G\) with this group,
using the convention from \Cref{sec:unitary-preliminaries}:
\[
 \chi_N(\Theta)
 =\exp\left(2\pi\mathrm{i}\sum_{j,k=1}^r n_{jk}\theta_{jk}\right),
\]
where \(N=(n_{jk})\in\Mat_r(\setZ)\) and
\(\Theta=(\theta_{jk})\in\Mat_r(\mathbb T)\).
For \(v\in V\) and \(w\in\setZ^r\), \(\chi_{vw^{\mathsf T}}(\Theta)=\exp(2\pi\mathrm{i}\,v^{\mathsf T}\Theta w)\).
Hence
\[
 \begin{aligned}
 H_v^\perp
 &=\Set*{\Theta\in\Mat_r(\mathbb T)\given
                 v^{\mathsf T}\Theta w=0\text{ for every }w\in\setZ^r}\\
 &=\Set*{\Theta\in\Mat_r(\mathbb T)\given v^{\mathsf T}\Theta=0}.
 \end{aligned}
\]
The last equality follows by taking \(w\) to be each standard
basis vector; the equalities are in \(\mathbb T\) and
\(\mathbb T^{1\times r}\), respectively.

Let \(\nu_\xi\) be the scalar spectral measure associated with
\(\xi\).
By \eqref{eq:spectral-fixed-space},
\begin{equation}\label{eq:matrix-fixed-projection}
 \|P_v\xi\|^2=\nu_\xi(H_v^\perp).
\end{equation}
Moreover,
\[
 \nu_\xi(\Mat_r(\mathbb T))=\|\xi\|^2,
 \qquad
 \nu_\xi(\{0\})=0,
\]
since \(\pi\) has no nonzero \(G\)-invariant vectors.

For \(\Theta\in\Mat_r(\mathbb T)\), let
\[
 V_\Theta
 =
 \Set*{
   v\in V
   \given
   v^{\mathsf T}\Theta=0
 }.
\]
We claim that
\begin{equation}\label{eq:cube-half}
 |V_\Theta|\leq2^{r-1}
 \qquad(\Theta\neq0).
\end{equation}
Indeed, if \(\Theta\neq0\), choose \(i\) such that \(e_i^{\mathsf T}\Theta\neq0\).
Partition \(V\) into the \(2^{r-1}\) pairs
\[
 \{v,v+e_i\},
 \qquad v_i=0.
\]
At most one member of each pair belongs to \(V_\Theta\), since
\[
 v^{\mathsf T}\Theta=(v+e_i)^{\mathsf T}\Theta=0 \Longrightarrow e_i^{\mathsf T}\Theta=0,
\]
contrary to the choice of \(i\).
This proves \eqref{eq:cube-half}.

Using \eqref{eq:matrix-fixed-projection} and finite summation,
we obtain
\[
\frac1{2^r}\sum_{v\in V}\|P_v\xi\|^2
 =
 \frac1{2^r}\sum_{v\in V}\nu_\xi(H_v^\perp)=
 \int_{\Mat_r(\mathbb T)}
 \frac{|V_\Theta|}{2^r}\,d\nu_\xi(\Theta).
\]
Since \(\nu_\xi(\{0\})=0\), \eqref{eq:cube-half} gives
\[
 \frac1{2^r}\sum_{v\in V}\|P_v\xi\|^2
 \leq
 \frac12
 \nu_\xi(\Mat_r(\mathbb T)\setminus\{0\})
 =
 \frac12\|\xi\|^2.
\]
By \eqref{eq:codistance-projection} and the arbitrariness of
\(\pi\in\operatorname{Rep}_0(G)\), this proves
\[
 \rho(\{H_v\}_{v\in V},G)\leq\frac12.
\]

Finally, since the subgroups \(H_v\) generate \(G\),
\eqref{eq:codistance-kazhdan} yields
\[
 \kappa(G,S)
 \geq
 \sqrt{2\left(1-\frac12\right)}=1.
\]
Therefore \((S,1)\) is a Kazhdan pair.
\end{proof}

Applying \Cref{lem:kazhdan-separation} with constant \(1\)
converts this Kazhdan pair into the displacement bound used below.

\begin{corollary}\label{cor:integer-matrix-displacement}
  Let \(r\geq1\), let \(\sigma:(\Mat_r(\setZ),+)\longrightarrow U(\mathcal H)\)
  be a unitary representation.
  Then, for every \(\xi\in\mathcal H\),
\begin{equation}\label{eq:matrix-conclusion}
 \sup_{N\in\Mat_r(\setZ)}
 \|\sigma(N)\xi-\xi\|
 \leq
 2
 \sup_{\substack{v\in\{0,1\}^r\\ w\in\setZ^r}}
 \|\sigma(vw^{\mathsf T})\xi-\xi\|.
\end{equation}
\end{corollary}

\begin{proof}
Apply \Cref{lem:kazhdan-separation} with \(\varepsilon=1\) to
the Kazhdan pair in \Cref{prop:integer-matrices-kazhdan}.
\end{proof}

\begin{remark}\label{rem:direct-averaging}
The Kazhdan pair in \Cref{prop:integer-matrices-kazhdan} can
also be obtained directly by averaging. Use \(G,V,S\) and
\(V_\Theta\) as in the proposition and its proof, and put
\(F_M=\{-M,\ldots,M\}^r\) for positive integers \(M\).
For \(\Theta\in\Mat_r(\mathbb T)\) and \(v\in V\), the product
of the one-dimensional character averages gives
\[
 \frac1{|F_M|}\sum_{w\in F_M}\overline{\chi_{vw^{\mathsf T}}(\Theta)}
 \longrightarrow
 \begin{cases}
  1,&v^{\mathsf T}\Theta=0,\\
  0,&v^{\mathsf T}\Theta\neq0.
 \end{cases}
\]
Let \(\pi:G\to U(\mathcal H)\) have no nonzero invariant
vectors, and let \(\eta\in\mathcal H\). Its spectral measure
satisfies \(\nu_\eta(\{0\})=0\) and
\(\nu_\eta(\Mat_r(\mathbb T))=\|\eta\|^2\).
Using \(|z-1|^2=2-2\Rea z\) for \(|z|=1\), the spectral
formula and dominated convergence give
\[
 \begin{aligned}
 &\lim_{M\to\infty}\frac1{2^r|F_M|}
   \sum_{v\in V}\sum_{w\in F_M}
   \|\pi(vw^{\mathsf T})\eta-\eta\|^2\\
 &\qquad=2\int_{\Mat_r(\mathbb T)}
       \left(1-\frac{|V_\Theta|}{2^r}\right)\,d\nu_\eta(\Theta)
 \geq\|\eta\|^2.
 \end{aligned}
\]
The averaged squared character differences are bounded by \(4\),
and the last inequality follows from \eqref{eq:cube-half}.
Each finite average is at most
\(\sup_{s\in S}\|\pi(s)\eta-\eta\|^2\), even though repeated
matrices may occur in the sum. Hence
\[
 \|\eta\|\leq\sup_{s\in S}\|\pi(s)\eta-\eta\|.
\]
Since this holds for every \(\pi\in\operatorname{Rep}_0(G)\), \((S,1)\) is a Kazhdan pair.
\end{remark}

\section{Elementary matrices and exoticness of the unit group}\label{sec:density}

In this section we prove that \(\GL(Q)\) is exotic.
The argument combines triviality on the elementary subgroup
with its density in \(\GL(Q)\).

\begin{theorem}\label{thm:roots}
Every strongly continuous unitary representation of
\(\GL(Q)\) is trivial on
\(\E_\mu(R)\).
\end{theorem}

\begin{proof}
Let \(\pi:\GL(Q)\to U(\mathcal H)\) be a strongly continuous
unitary representation. Fix \(k\) with \(n=n_k\geq2\), distinct
indices \(i,j\in\Set{1,\ldots,n}\), and \(a\in R\), and let
\[
 u=1_n+ae_{ij}\in\E_n(R).
\]
It suffices to prove that \(\pi(\iota_k(u))=\id_{\mathcal H}\),
since these elements generate \(\E_\mu(R)\).

For each \(\ell>k\), set \(r=n_\ell/n\) and define the additive map \(T_r:\Mat_r(\setZ)\to\Mat_{nr}(R)\) by
\[
 T_r(N)=N\otimes(ae_{ij}).
\]
Since \(i\ne j\), we have \((ae_{ij})^2=0\). Integer coefficients
are central in \(R\), so for all \(N,N'\in\Mat_r(\setZ)\),
\[
 T_r(N)T_r(N')=NN'\otimes(ae_{ij})^2=0.
\]
Consequently,
\[
 (1_{nr}+T_r(N))(1_{nr}+T_r(N'))=1_{nr}+T_r(N+N'),
\]
and \(1_{nr}+T_r(N)\) is invertible with inverse
\(1_{nr}-T_r(N)\).
Hence
\[
\sigma_r(N)=\pi\bigl(\iota_\ell(1_{nr}+T_r(N))\bigr).
\]
defines a unitary representation of the discrete additive group \(M_r(\mathbb Z)\).

For \(v,w\in\setZ^r\), the factorization
\[
 T_r(vw^{\mathsf T})
 =(v\otimes ae_i)(w^{\mathsf T}\otimes e_j^{\mathsf T})
\]
expresses this matrix as the product of an \(nr\times1\) column
and a \(1\times nr\) row. The Sylvester rank axioms therefore give \(\rk(T_r(vw^{\mathsf T}))\leq1\), and hence
\[
 d\bigl(\iota_\ell(1_{nr}+T_r(vw^{\mathsf T})),1_Q\bigr)
 =\rk_\mu\bigl(\iota_\ell(T_r(vw^{\mathsf T}))\bigr)=\rk_{nr}(T_r(vw^{\mathsf T}))\leq\frac1{nr}.
\]
This bound is uniform in \(v,w\), regardless of their integer entries.

Fix \(\xi\in\mathcal H\).
For \(t>0\), set
\[
 \beta_\xi(t)=\sup\Set*{\|\pi(g)\xi-\xi\|\given
             g\in\GL(Q),\ d(g,1_Q)\leq t}.
\]
Unitarity gives \(0\leq\beta_\xi(t)\leq2\|\xi\|\), and strong
continuity at the identity implies \(\beta_\xi(t)\to0\) as
\(t\downarrow0\). The preceding rank bound gives
\[
 \|\sigma_r(vw^{\mathsf T})\xi-\xi\|
 \leq\beta_\xi\bigl((nr)^{-1}\bigr)
 \qquad(v,w\in\setZ^r).
\]
Applying \Cref{cor:integer-matrix-displacement} to \(\sigma_r\) therefore yields
\[
 \sup_{N\in\Mat_r(\setZ)}\|\sigma_r(N)\xi-\xi\|
 \leq2\beta_\xi\bigl((nr)^{-1}\bigr).
\]
On the other hand, evaluating at \(1_r\) gives the image of \(u\)
under \(\phi_{\ell,k}\):
\[
 1_{nr}+T_r(1_r)=1_r\otimes(1_n+ae_{ij})
 =\phi_{\ell,k}(u).
\]
By \eqref{eq:canonical-compatibility},
\[
 \sigma_r(1_r)=\pi\bigl(\iota_\ell(\phi_{\ell,k}(u))\bigr)
 =\pi(\iota_k(u)).
\]
Since \(nr=n_\ell\to\infty\), we conclude that
\[
 \|\pi(\iota_k(u))\xi-\xi\|
 \leq2\beta_\xi(n_\ell^{-1})\longrightarrow0.
\]
The left-hand side is independent of \(\ell\), so it is zero.
As \(\xi\) was arbitrary, \(\pi(\iota_k(u))=\id_{\mathcal H}\).
This proves that \(\pi\) is trivial on \(\E_\mu(R)\).
\end{proof}


To pass from \(\E_\mu(R)\) to \(\GL(Q)\), it remains to prove
density.
The required approximation follows from a quantitative form of the Whitehead block identity
\cite[Example~III.1.2.1]{WeibelKbook}:
an approximate right inverse for \(X\) yields, after amplification, a nearby
element of the elementary subgroup.

\begin{lemma}\label{lem:correction}
  Let \((R,\rk)\) be a rank ring.
For \(X,Y\in\Mat_n(R)\) and integer \(r\geq2\), there is
\(U\in\E_{nr}(R)\) with
\[
 \rk_{nr}(1_r\otimes X-U)
 \leq2(1-1/r)\rk_n(1_n-XY)+\frac1r.
\]
\end{lemma}
\begin{proof}
  The identities
  \begin{align*}
    \begin{pmatrix}1_n&0\\-Y&1_n\end{pmatrix}
 \begin{pmatrix}1_n&X\\Y&1_n\end{pmatrix}\begin{pmatrix}1_n&-X\\0&1_n\end{pmatrix}
    & =\diag(1_n,1_n-YX),\\
    \begin{pmatrix}1_n&-X\\0&1_n\end{pmatrix}
 \begin{pmatrix}1_n&X\\Y&1_n\end{pmatrix}\begin{pmatrix}1_n&0\\-Y&1_n\end{pmatrix}
 & =\diag(1_n-XY,1_n)
  \end{align*}
give, by invariance under elementary multiplication and block additivity,
\[
 \rk(1_n-XY)=\rk(1_n-YX).
\]
Consider the elementary block product
\[
 W=
 \begin{pmatrix}1_n&X\\0&1_n\end{pmatrix}
 \begin{pmatrix}1_n&0\\-Y&1_n\end{pmatrix}
 \begin{pmatrix}1_n&X\\0&1_n\end{pmatrix}
 \begin{pmatrix}0&-1_n\\1_n&0\end{pmatrix}
 =
 \begin{pmatrix}2X-XYX&XY-1_n\\1_n-YX&Y\end{pmatrix}\in \E_{2n}(R).
\]
Each upper block matrix is a product of the elementary matrices
\(1_{2n}+x_{ij}e_{i,n+j}\); the lower block matrix is treated in
the same way, using the entries of \(-Y\).
The off-diagonal block matrix also belongs to \(\E_{2n}(R)\), since
\[
 \begin{pmatrix}0&-1_n\\1_n&0\end{pmatrix}
 =\begin{pmatrix}1_n&-1_n\\0&1_n\end{pmatrix}
  \begin{pmatrix}1_n&0\\1_n&1_n\end{pmatrix}
  \begin{pmatrix}1_n&-1_n\\0&1_n\end{pmatrix}.
\]
We have
\[
 W-\diag(X,Y)
 =\begin{pmatrix}X-XYX&XY-1_n\\1_n-YX&0\end{pmatrix}
 =\begin{pmatrix}1_n-XY&0\\0&1_n-YX\end{pmatrix}
  \begin{pmatrix}X&-1_n\\1_n&0\end{pmatrix},
\]
the product inequality and block additivity give
\[
 \rk(W-\diag(X,Y))
 \leq\rk(1_n-XY)+\rk(1_n-YX)=2\rk(1_n-XY).
\]
For \(1\leq h<r\),
define
\begingroup
\setlength{\arraycolsep}{4pt}
\[
 W_h=
 \begin{pmatrix}
  1_{(h-1)n}&0&0&0\\
  0&2X-XYX&0&XY-1_n\\
  0&0&1_{(r-h-1)n}&0\\
  0&1_n-YX&0&Y
 \end{pmatrix},\quad
 D_h=
 \begin{pmatrix}
  1_{(h-1)n}&0&0&0\\
  0&X&0&0\\
  0&0&1_{(r-h-1)n}&0\\
  0&0&0&Y
 \end{pmatrix}.
\]
\endgroup
Each zero block has the size determined by its block row and column.
Blocks of size zero are omitted when \(h=1\) or \(h=r-1\).
The elementary factorization of \(W\), applied to block indices
\(h,r\), gives \(W_h\in\E_{nr}(R)\).
Moreover,
\[
 W_h-D_h = \begin{pmatrix}
  0_{(h-1)n}&0&0&0\\
  0&X-XYX&0&XY-1_n\\
  0&0&0_{(r-h-1)n}&0\\
  0&1_n-YX&0&0_n
 \end{pmatrix}.
\]
Hence
\[
 \rk(W_h-D_h)=\rk(W-\diag(X,Y))\leq2\rk(1_n-XY).
\]
Set \(U=W_1\cdots W_{r-1}\), which belongs to \(\E_{nr}(R)\).
The identity
\[
 U-D_1\cdots D_{r-1}
 =\sum_{h=1}^{r-1}W_1\cdots W_{h-1}(W_h-D_h)D_{h+1}\cdots D_{r-1}
\]
is valid in this order over a noncommutative ring, with empty products
equal to the identity.
Rank subadditivity and the product inequality give
\[
 \rk(U-D_1\cdots D_{r-1})
 \leq\sum_{h=1}^{r-1}\rk(W_h-D_h)
 \leq2(r-1)\rk(1_n-XY).
\]
Finally,
\[
 D_1\cdots D_{r-1}
 =\diag(\underbrace{X,\ldots,X}_{r-1\text{ copies}},Y^{r-1}),
\]
so
\[
 \rk(1_r\otimes X-D_1\cdots D_{r-1})
 =\rk(X-Y^{r-1})\leq n.
\]
Consequently,
\[
 \begin{aligned}
 \rk_{nr}(1_r\otimes X-U)
 &\leq\frac{\rk(1_r\otimes X-D_1\cdots D_{r-1})
                +\rk(D_1\cdots D_{r-1}-U)}{nr}\\
 &\leq\frac{n+2(r-1)\rk(1_n-XY)}{nr}\\
 &=2(1-1/r)\rk_n(1_n-XY)+\frac1r.
 \end{aligned}
\]
\end{proof}

For a unit \(g\in\GL(Q)\), simultaneous approximations to \(g\)
and \(g^{-1}\) make the defect \(1_n-XY\) arbitrarily small.
The remaining term \(1/r\) in \Cref{lem:correction} tends to zero
under further amplification, giving density.

\begin{proposition}\label{prop:density}
The subgroups \(\E_\mu(R)\) and \(\GL_\mu(R)\) are dense in
\(\GL(Q)\). More precisely,
\[
 \overline{\E_\mu(R)}
 =\overline{\GL_\mu(R)}=\GL(Q),
\]
where closures are taken in the rank metric on \(Q\).
\end{proposition}
\begin{proof}
Fix \(g\in\GL(Q)\) and \(0<\varepsilon<1/2\).
By density of \(\bigcup_{k\geq0}\iota_k(\Mat_{n_k}(R))\subset Q\),
choose \(k\geq0\), put \(n=n_k\), and take
\(X,Y\in\Mat_n(R)\), writing
\(x=\iota_k(X)\) and \(y=\iota_k(Y)\), such that
\[
 d(g,x)<\varepsilon,\qquad d(g^{-1},y)<\varepsilon.
\]
The identity
\[
 1_Q-xy=(g-x)g^{-1}+x(g^{-1}-y)
\]
and \eqref{eq:canonical-rank} give
\[
 \rk_n(1_n-XY)
 =\rk_\mu(1_Q-xy)
 \leq d(g,x)+d(g^{-1},y)<2\varepsilon.
\]

Since \(\mu\) is unbounded, choose \(\ell>k\) such that
\(r=n_\ell/n\geq2\) and \(1/r\leq\varepsilon\).
By \Cref{lem:correction}, there is \(U\in\E_{nr}(R)\) with
\(\rk_{nr}(1_r\otimes X-U)\leq5\varepsilon\).

Now \(\iota_\ell(U)\in\E_\mu(R)\), and
\(\iota_\ell(1_r\otimes X)=\iota_k(X)=x\). Hence
\[
 d(g,\iota_\ell(U))
 \leq d(g,x)+d(x,\iota_\ell(U))=d(g,x)+\rk_{nr}(1_r\otimes X-U)
 <6\varepsilon.
\]
As \(\varepsilon\) is arbitrary,
\(\GL(Q)\subseteq\overline{\E_\mu(R)}\).

For the reverse inclusion, let \((u_j)\) be a sequence in
\(\GL(Q)\) converging to \(u\in Q\). Multiplication by units
preserves rank, so
\[
 \begin{aligned}
 d(u_j^{-1},u_m^{-1})
 &=\rk_\mu\bigl(u_j^{-1}(u_m-u_j)u_m^{-1}\bigr)\\
 &=\rk_\mu(u_m-u_j)=d(u_j,u_m)\longrightarrow0
 \quad(j,m\to\infty).
 \end{aligned}
\]
Thus \((u_j^{-1})\) is Cauchy. Completeness of \(Q\) gives
\(v\in Q\) with \(u_j^{-1}\to v\), and continuity of multiplication gives
\[
 uv=\lim_{j\to\infty}u_ju_j^{-1}=1_Q,\qquad
 vu=\lim_{j\to\infty}u_j^{-1}u_j=1_Q.
\]
Hence \(u\in\GL(Q)\) with \(u^{-1}=v\), proving that
\(\GL(Q)\) is closed in \(Q\). Since
\(\E_\mu(R)\subseteq\GL_\mu(R)\subseteq\GL(Q)\), it follows that
\[
 \overline{\E_\mu(R)}\subseteq\overline{\GL_\mu(R)}\subseteq\GL(Q).
\]
Together with the first inclusion, this proves the assertion.
\end{proof}

Combining \Cref{thm:roots,prop:density} gives the unit-group result.

\begin{theorem}\label{thm:main}
\MainTheoremStatement
\end{theorem}
\begin{proof}
The density assertion is \Cref{prop:density}.
By \Cref{thm:roots}, every representation of \(\GL(Q)\) is
trivial on \(\E_\mu(R)\). Its kernel is closed and contains this
dense subgroup, so it equals \(\GL(Q)\).
\end{proof}

\section{The additive group}\label{sec:additive}

In the proof of \Cref{thm:roots}, square-zero perturbations allowed us to encode
addition as multiplication in the unit group.
For \((Q,+)\), the map
\(N\mapsto N\otimes X\) is already additive for every matrix
\(X\), so the same amplification argument applies to all
finite-level matrices.

\begin{theorem}\label{thm:additive}
The additive group \((Q,+)\) is exotic.
\end{theorem}
\begin{proof}
Let \(\pi:(Q,+)\to U(\mathcal H)\) be a representation.
Fix \(k\geq0\), put \(n=n_k\), and take \(X\in\Mat_n(R)\).
For \(\ell>k\), put \(r=n_\ell/n\) and define
\[
 \sigma_r:(\Mat_r(\setZ),+)\longrightarrow U(\mathcal H),
 \qquad \sigma_r(N)=\pi\bigl(\iota_\ell(N\otimes X)\bigr).
\]
This is a representation of the discrete additive group
\((\Mat_r(\setZ),+)\): for \(N=(n_{pq})\in\Mat_r(\setZ)\),
\(N\otimes X=(n_{pq}X)_{p,q}\) uses only integer scalar multiples.
For \(v,w\in\setZ^r\), the factorization
\[
 (vw^{\mathsf T})\otimes X
   =(v\otimes 1_n)X(w^{\mathsf T}\otimes 1_n)
\]
and \eqref{eq:canonical-rank} give
\[
 \rk_\mu\bigl(\iota_\ell((vw^{\mathsf T})\otimes X)\bigr)
 =\rk_{nr}((vw^{\mathsf T})\otimes X)
 \leq\frac{\rk(X)}{nr}\leq\frac1r.
\]
For a fixed \(\xi\in\mathcal H\) and \(t>0\), put
\[
 \gamma_\xi(t)=\sup\Set*{\|\pi(x)\xi-\xi\|\given x\in Q,\ \rk_\mu(x)\leq t}.
\]
Strong continuity at zero gives \(\gamma_\xi(t)\to0\) as
\(t\downarrow0\), and the preceding rank bound gives
\[
 \sup_{\substack{v\in\{0,1\}^r\\w\in\setZ^r}}
       \|\sigma_r(vw^{\mathsf T})\xi-\xi\|
 \leq\gamma_\xi(1/r).
\]
Moreover,
\[
 \iota_\ell(1_r\otimes X)
 =\iota_\ell(\phi_{\ell,k}(X))=\iota_k(X).
\]
Thus \Cref{cor:integer-matrix-displacement} yields
\[
 \|\pi(\iota_k(X))\xi-\xi\|
 =\|\sigma_r(1_r)\xi-\xi\|
 \leq2\gamma_\xi(1/r)\longrightarrow0.
\]
Here \(\ell\to\infty\), so \(r=n_\ell/n\to\infty\), whereas
the left side is independent of \(\ell\). It is therefore zero
for every \(\xi\), proving \(\pi(\iota_k(X))=\id_{\mathcal H}\).
The union of the images \(\iota_k(\Mat_{n_k}(R))\) is dense in
\(Q\), so the representation is trivial.
\end{proof}

\section{Small subgroups and escape functions}\label{sec:small-subgroups}

For \(\varepsilon>0\) and an identity neighborhood \(U\) in
\(\GL(Q)\), put
\[
 B_\varepsilon=\Set*{g\in\GL(Q)\given d(g,1_Q)<\varepsilon},
 \qquad
 T(U)=\bigcup\Set*{H\leq\GL(Q)\given H\subseteq U}.
\]
The set \(T(U)\) need not itself be a subgroup. Products such as
\(T(U)^m\) denote products of \(m\) elements of this set.

The small-support factorization principle of \cite[Lemma~3.8 and
Proposition~3.9]{SchneiderSolecki} can be implemented here using diagonal
corners.
If \((S,\rk)\) is a rank ring and \(e\in S\) is
an idempotent, then
\[
 \Gamma_S(e):=\Set*{a+1_S-e\given a\in\GL(eSe)}\leq\GL(S),
 \qquad \rk(h-1_S)\leq\rk(e)\quad(h\in\Gamma_S(e)).
\]
Indeed, \((a+1_S-e)(b+1_S-e)=ab+1_S-e\) for \(a,b\in eSe\),
and the inverse of \(a+1_S-e\) is \(a^{-1}+1_S-e\), where
\(a^{-1}\) is taken in the corner ring with identity \(e\).
The rank bound follows from \(h-1_S=a-e=e(a-e)e\);
compare \cite[Lemma~6.1]{SchneiderGeometry}.

\begin{proposition}\label{prop:small-subgroup-width}
For every \(\varepsilon>0\) and every positive integer
\(m>\varepsilon^{-1}\),
\begin{equation}\label{eq:small-subgroup-width}
 \GL_\mu(R)\subseteq T(B_\varepsilon)^m.
\end{equation}
\end{proposition}
\begin{proof}
Fix \(g=\iota_k(u)\) with \(u\in\GL_{n_{k}}(R)\).
Since \(1/m<\varepsilon\) and \(n_\ell/n_k\to\infty\), choose
\(\ell\geq k\) such that \(r=n_\ell/n_{k}\geq m\) and
\[
 \frac1m+\frac1r<\varepsilon.
\]
Partition \(\{1,\ldots,r\}\) into nonempty sets
\(I_1,\ldots,I_m\) with \(|I_s|\leq\lceil r/m\rceil\).
In \(\Mat_{rn_{k}}(R)\), let \(p_s\) be the diagonal block
idempotent with block \(1_{n_{k}}\) at positions in \(I_s\) and
block \(0_{n_{k}}\) elsewhere. Set
\[
 A=1_r\otimes u,\qquad
 v_s=1_{rn_k}-p_s+p_sAp_s.
\]
The matrix \(A\) commutes with each \(p_s\), and its inverse is
\(1_r\otimes u^{-1}\). Hence
\[
 (p_sAp_s)(p_sA^{-1}p_s)
 =(p_sA^{-1}p_s)(p_sAp_s)=p_s.
\]
Thus \(p_sAp_s\) is a unit in the corner with identity \(p_s\),
and \(v_s\) is a unit in \(\Mat_{rn_k}(R)\).

Put \(e_s=\iota_\ell(p_s)\).
These are idempotents of \(Q\).
Applying \(\iota_\ell\) to \(p_sAp_s\) and \(p_sA^{-1}p_s\) shows that
\(\iota_\ell(p_sAp_s)\) is invertible in \(e_sQe_s\), with inverse
\(\iota_\ell(p_sA^{-1}p_s)\). In particular,
\[
 \iota_\ell(v_s)\in\Gamma_Q(e_s).
\]
This argument does not require \(\iota_\ell\) to be injective.
By block additivity and \eqref{eq:canonical-rank},
\[
 \rk_\mu(e_s)
 =\frac{\rk(p_s)}{rn_{k}}
 =\frac{|I_s|n_{k}}{n_\ell}
 =\frac{|I_s|}{r}
 \leq\frac{\lceil r/m\rceil}{r}
 \leq\frac1m+\frac1r<\varepsilon.
\]
The bound now gives
\(\Gamma_Q(e_s)\subseteq B_\varepsilon\), so
\(\iota_\ell(v_s)\in T(B_\varepsilon)\).
The matrices \(v_s\) act on disjoint groups of blocks; their
product is exactly \(A\). Therefore
\[
 g=\iota_\ell(1_r\otimes u)
   =\prod_{s=1}^m\iota_\ell(v_s)\in T(B_\varepsilon)^m.
\]
The choices of \(m\) and \(\varepsilon\) were independent of
\(k\) and \(u\), proving \eqref{eq:small-subgroup-width}.

\end{proof}

By \Cref{prop:density}, \(\GL_\mu(R)\) is dense in \(\GL(Q)\).
For the same \(\varepsilon\) and \(m\), this gives the first equality in
\begin{equation}\label{eq:ball-width}
 \GL(Q)=\overline{T(B_\varepsilon)^m}=B_\varepsilon^m,
\end{equation}
where the closure is taken in \(\GL(Q)\).
For the second equality, choose \(\eta\) with
\(1/m<\eta<\varepsilon\).
For any \(g\in\GL(Q)\), choose \(h\in\GL_\mu(R)\) with
\(d(g,h)<\varepsilon-\eta\). Applying
\eqref{eq:small-subgroup-width} with radius \(\eta\), write
\(h=h_1\cdots h_m\), where \(h_s\in T(B_\eta)\subseteq B_\eta\).
Bi-invariance and the triangle inequality give
\[
 d((gh^{-1})h_1,1_Q)
 \leq d(gh^{-1},1_Q)+d(h_1,1_Q)<\varepsilon.
\]
Thus \(g=((gh^{-1})h_1)h_2\cdots h_m\in B_\varepsilon^m\),
proving the remaining equality, including when \(m=1\).

This argument does not establish \(\GL(Q)=T(B_\varepsilon)^m\): the
corrected first factor belongs to the ball, but need not belong
to a subgroup contained in that ball.

A topological group \(G\) is \emph{Bourbaki bounded} if for every
identity neighborhood \(U\) there are a finite set \(F\subseteq G\)
and a positive integer \(N\) such that \(G=FU^N\).
Equivalently, \(G\) is bounded with respect to its left uniformity; in our
bi-invariant metric setting, the left and right uniformities coincide.

\begin{corollary}\label{cor:bourbaki}
For every \(\varepsilon>0\), the least positive integer \(N\) such that
\(\GL(Q)=B_\varepsilon^N\) is
\[
 N=\lfloor\varepsilon^{-1}\rfloor+1.
\]
In particular, \(\GL(Q)\) is Bourbaki bounded.
\end{corollary}
\begin{proof}
The upper bound follows by taking
\(m=\lfloor\varepsilon^{-1}\rfloor+1\) in \eqref{eq:ball-width}.

For optimality, the integer matrices
\[
 A_2=\begin{pmatrix}0&-1\\1&1\end{pmatrix},\qquad
 A_3=\begin{pmatrix}0&0&1\\1&0&1\\0&1&0\end{pmatrix}
\]
satisfy
\(\det A_j=\det(A_j-1_j)=1\) for \(j=2,3\).
For any \(n=n_k\geq2\), write \(n=2a+3b\) with nonnegative
integers \(a,b\), and form the block diagonal matrix \(A\) with
\(a\) copies of \(A_2\) and \(b\) copies of \(A_3\).
Both \(A\) and \(A-1_n\) are invertible over \(\setZ\), hence
over \(R\) via its canonical coefficient map. Consequently
\(g=\iota_k(A)\) is a unit with \(d(g,1_Q)=1\).
Every product of \(N\) elements of \(B_\varepsilon\) has
distance strictly less than \(N\varepsilon\) from \(1_Q\).
Thus \(\GL(Q)=B_\varepsilon^N\) requires \(N\varepsilon>1\),
which proves the asserted minimum. Every identity neighborhood
contains such a ball, so Bourbaki boundedness follows with
\(F=\{1_Q\}\).
\end{proof}

The stronger factorization through small subgroups is needed for escape
functions, since if \(H\leq \GL(Q)\) is contained in an identity neighborhood
\(U\), then every positive power of each \(h\in H\) remains in \(U\).
Following \cite[Definition~3.1]{SchneiderSolecki}, a \emph{length function}
on a group \(G\) is a function \(f:G\to[0,\infty)\) satisfying the
following conditions for all \(g,h\in G\):
\[
 (1) f(1_G)=0,\qquad (2) f(g^{-1})=f(g),\qquad (3)
 f(gh)\leq f(g)+f(h).
\]
If \(G\) is a topological group, \(f\) is an \emph{escape function}
if some identity neighborhood \(U\) has the following property:
for every \(\delta>0\), there is a positive integer \(N\) such that,
for all \(g\in G\),
\begin{equation}\label{eq:escape-definition}
 g,g^2,\ldots,g^N\in U\quad\Longrightarrow\quad f(g)<\delta.
\end{equation}
Such a \(U\) is called an \emph{escape neighborhood}.
Every escape function is continuous. Indeed, for this \(N\) the
set \(\bigcap_{j=1}^N \Set*{g\in G\given g^j\in U}\) is an identity
neighborhood on which \(f<\delta\).
This proves continuity at the identity; continuity everywhere follows
from \(\lvert f(g)-f(h)\rvert\leq f(h^{-1}g)\).

\begin{corollary}\label{cor:no-escape}
Every escape function on \(\GL(Q)\) is zero.
\end{corollary}
\begin{proof}
Let \(f\) be an escape function with escape neighborhood \(U\).
Choose \(\varepsilon>0\) with \(B_\varepsilon\subseteq U\),
and a positive integer \(m>\varepsilon^{-1}\).
If \(H\leq\GL(Q)\) lies in \(B_\varepsilon\), all positive powers
of each \(h\in H\) lie in \(U\). By
\eqref{eq:escape-definition}, \(f(h)<\delta\) for every
\(\delta>0\), and hence \(f(h)=0\).
Thus \(f\) vanishes on \(T(B_\varepsilon)\), and subadditivity
makes it vanish on \(T(B_\varepsilon)^m\).
By \eqref{eq:ball-width} this set is dense in \(\GL(Q)\), so
continuity gives \(f=0\).
\end{proof}

The preceding proof is the mechanism of
\cite[Lemma~6.6]{SchneiderGeometry}. It also controls continuous
homomorphisms by pulling back escape functions.
A topological group \(H\) has the \emph{escape property} if each
identity neighborhood \(V\) contains a set \(f^{-1}([0,1))\)
for some escape function \(f\) on \(H\).

\begin{corollary}\label{cor:escape-homomorphisms}
Let \(H\) be a Hausdorff topological group with the escape
property. Every continuous homomorphism \(\psi:\GL(Q)\to H\)
is trivial.
\end{corollary}
\begin{proof}
As in \cite[Lemma~6.7]{SchneiderGeometry}, the pullback of an
escape function \(f\) on \(H\) is an escape function on \(\GL(Q)\):
the inverse image under \(\psi\) of an escape neighborhood
for \(f\) verifies \eqref{eq:escape-definition} for \(f\circ\psi\).
Thus \(f\circ\psi=0\) by \Cref{cor:no-escape}.
For each identity neighborhood \(V\) in \(H\), choose \(f\)
with \(f^{-1}([0,1))\subseteq V\). Then \(\psi(\GL(Q))\subseteq V\).
The intersection of all identity neighborhoods in a Hausdorff
group is \(\{1_H\}\), proving the assertion.
\end{proof}

\section{Extreme amenability and examples}\label{sec:applications}

\subsection{Fixed-point consequences}

A topological group \(G\) is \emph{amenable} if every continuous
action of \(G\) on a nonempty compact Hausdorff space \(X\) admits
an invariant regular Borel probability measure. It is \emph{extremely
amenable} if every such action has a global fixed point, that is,
a point \(x\in X\) such that \(gx=x\) for all \(g\in G\).
Amenability provides an invariant measure; exoticness then forces
its support to consist of fixed points.

\begin{lemma}[{\cite[Remark~5.7]{SchneiderSolecki}}]\label{lem:dynamics}
Every amenable exotic topological group is extremely amenable.
\end{lemma}

Since \((Q,+)\) is abelian, it is amenable. Its exoticness,
proved in \Cref{thm:additive}, therefore gives the following
consequence of \Cref{lem:dynamics}.

\begin{theorem}\label{thm:extreme-amenability}
The additive group \((Q,+)\) is extremely amenable.
\end{theorem}

For \(\GL(Q)\), amenability follows under an additional hypothesis on the
finite-level matrix groups.
Requiring all these discrete groups
to be amenable is equivalent to \(R\) being \emph{locally finite},
that is, every finite subset of \(R\) lies in a finite subring
with the same identity.

The implication from amenability to local finiteness uses
\cite[Theorem~1.1]{EJ}.
We also use the standard facts that amenability passes to subgroups, quotients, and
directed unions, and that finite groups are amenable \cite{JuschenkoAmenability,PierAmenable}.

\begin{theorem}\label{thm:unit-extreme-amenability}
The following conditions on \(R\) are equivalent, with the matrix
groups equipped with the discrete topology:
\begin{enumerate}
\item \(R\) is locally finite;
\item \(\GL_n(R)\) is amenable for every \(n\geq1\);
\item \(\GL_{n_k}(R)\) is amenable for every \(k\);
\item \(\E_3(R)\) is amenable.
\end{enumerate}
Under these conditions, \(\GL(Q)\), endowed with the rank topology,
is extremely amenable.
\end{theorem}
\begin{proof}
If \(R\) is locally finite, any finite subset of \(\GL_n(R)\) is
contained in \(\GL_n(S)\) for some finite unital subring
\(S\subseteq R\): take \(S\) containing the entries of the matrices
and their inverses. Hence \(\GL_n(R)\) is a locally finite group,
so (1) implies (2). Clearly (2) implies (3). Since \((n_k)\) is
unbounded, choose \(k\) with \(n_k>3\). The embedding
\(A\mapsto\diag(A,1_{n_k-3})\) identifies \(\E_3(R)\) with a
subgroup of \(\GL_{n_k}(R)\), proving (3) implies (4).

Assume (4), and let \(S\subseteq R\) be a finitely generated
subring containing \(1_R\). The group \(\E_3(S)\) is amenable as a
subgroup of \(\E_3(R)\), and has Property~\textup{(T)} by
\cite[Theorem~1.1]{EJ}. It is therefore finite by
\cite[Theorem~1.1.6]{BekkaHarpeValette}. The injection
\(s\mapsto1_3+se_{12}\) from \(S\) into \(\E_3(S)\) shows that
\(S\) is finite. This proves (4) implies (1).

Under these conditions, \(\GL_\mu(R)\) is amenable as a discrete
group, being the increasing union of the homomorphic images
\(\iota_k(\GL_{n_k}(R))\). Thus every continuous action of
\(\GL(Q)\) on a nonempty compact Hausdorff space admits a
\(\GL_\mu(R)\)-invariant regular Borel probability measure.
Density of \(\GL_\mu(R)\), proved in \Cref{prop:density}, and
continuity of the action make this measure \(\GL(Q)\)-invariant:
for each continuous function on the compact space, its integral
against the translated measure depends continuously on the group
element. Now apply \Cref{thm:main,lem:dynamics}.
\end{proof}

Algebraic extensions of finite fields and matrix rings over locally
finite rings are locally finite.
Thus \Cref{thm:unit-extreme-amenability} applies to infinite rings as well.
For a finitely generated unital ring, local finiteness is equivalent
to finiteness.

\subsection{Examples of coefficient rings}\label{sec:coefficient-examples}

The following examples illustrate the conclusions for non-regular
completions and for completions with nontrivial central idempotents.
First take \(R=\setZ/4\setZ\), equipped with the Sylvester matrix
rank function
\[
 \rk_{\mathrm{len}}(A)=\frac12\ell_R(\operatorname{im} A)
 \qquad(A\in\Mat_{r\times s}(R)),
\]
where \(A:R^{s\times1}\to R^{r\times1}\) acts on column vectors
and \(\ell_R\) denotes composition length; see
\cite[Example~5.6]{JiangBratteli}.
For this rank, \(Q=\overline{R_\mu}\) is neither von Neumann
regular nor simple by \cite[Remark~5.7]{JiangBratteli}.
Moreover, \(\rk_\mu(2\cdot1_Q)=1/2\), so \(Q\) has
characteristic \(4\) and no unital field subring.
Nevertheless, \Cref{thm:main,cor:bourbaki,cor:no-escape} show that \(\GL(Q)\)
is exotic, Bourbaki bounded, and admits no nonzero escape function.
Since \(R\) is finite, \Cref{thm:unit-extreme-amenability} also
gives extreme amenability.

Non-regular completions also occur in characteristic zero.
Take \(R=\setZ\) with \(\rk(A)=\operatorname{rank}_{\mathbb Q}(A)\).
As in \cite[Example~5.8]{JiangBratteli}, the fact that
\(1_n-2Y\) has odd determinant for every integer matrix \(Y\)
gives, by passage to the completion,
\[
 \rk_\mu(1_Q-2x)=1\qquad(x\in Q).
\]
Multiplication by \(2\) preserves rational matrix rank, so
\(\rk_\mu(2z)=\rk_\mu(z)\) for all \(z\in Q\).
Consequently, \(\rk_\mu(2\cdot1_Q-4x)=1\) for every \(x\in Q\).
Thus \(2\cdot1_Q\) has no inner inverse, and \(Q\) is not
von Neumann regular.
As in the preceding example, \Cref{thm:main,cor:bourbaki,cor:no-escape} imply
that \(\GL(Q)\) is exotic, Bourbaki bounded, and admits no
nonzero escape function.
Here \(R\) is not locally finite, so the sufficient criterion
in \Cref{thm:unit-extreme-amenability} does not apply;
our arguments do not determine whether \(\GL(Q)\) is extremely amenable.

To obtain completions with nontrivial central idempotents, consider
a product of rank rings \((R_1,\rk^{(1)})\) and \((R_2,\rk^{(2)})\).
For \(0<t<1\), equip \(R=R_1\times R_2\) with the Sylvester matrix rank
\[
 \rk(A_1,A_2)=t\rk^{(1)}(A_1)+(1-t)\rk^{(2)}(A_2).
\]
Let \(Q_i\) be the completion associated with \((R_i,\rk^{(i)})\)
and the same factor sequence \(\mu\).
The matrix systems split as products, and a sequence is Cauchy
for the weighted rank metric precisely when both components are
Cauchy. Hence the completion \(Q\) is \(Q_1\times Q_2\), with
the corresponding weighted rank; compare
\cite[Example~5.12]{JiangBratteli}.
The central idempotent \((1_{Q_1},0)\) has rank \(t\), so this
completion is reducible. Its unit group still satisfies
\Cref{thm:main,cor:bourbaki,cor:no-escape}.

If \(R_1\) and \(R_2\) are finite fields then
\(\GL(Q)=\GL(Q_1)\times\GL(Q_2)\) is extremely amenable by
\Cref{thm:unit-extreme-amenability}, while its quotient by the
centre is not topologically simple.
Indeed, each \(\GL(Q_i)\) is nonabelian: for \(n=n_k\geq3\),
the commutator of the images of \(1_n+e_{12}\) and \(1_n+e_{23}\)
is the image of \(1_n+e_{13}\), whose distance from the identity
is \(1/n>0\).
The quotient of \(\GL(Q)\) by its centre is therefore the
product of two nontrivial Hausdorff groups, each defining a
proper closed normal subgroup.

\subsection{Weighted Bratteli completions}\label{sec:bratteli}

The uniqueness theorem of \cite{JiangBratteli} transfers the
results for factor-sequence completions to weighted Bratteli
completions under the hypotheses below. We use the notation of
that paper.
For a unital Bratteli diagram \(B\), let \(V_n\) be its finite
vertex set at level \(n\), let \(p_n(v)\) be the block size at
\(v\), and put
\(A_n(B,R)=\prod_{v\in V_n}\Mat_{p_n(v)}(R)\).
A harmonic function \(\alpha\) assigns compatible probability
weights \(\alpha_n(v)\) to the blocks, defining the rank
\[
 \rk_\alpha((X_v)_{v\in V_n})
 =\sum_{v\in V_n}\frac{\alpha_n(v)}{p_n(v)}\rk(X_v).
\]
The same formula applies to rectangular matrices using their
rectangular blocks. Let \(\overline A_\alpha(B,R)\) denote the
separated rank completion of the algebraic direct limit.

Suppose that \(\alpha\) is extreme and, for every positive integer
\(M\),
\[
 \lim_{n\to\infty}
 \sum_{\substack{v\in V_n\\p_n(v)\leq M}}\alpha_n(v)=0.
\]
By \cite[Theorem~1.1]{JiangBratteli}, this completion is isomorphic,
with its specified rank, to the matrix completion in
\Cref{sec:factor-sequences} for the dyadic sequence \((2^k)_{k\geq1}\).
The isomorphism restricts to isometric isomorphisms of the additive
and unit groups. Hence
\Cref{thm:main,thm:additive,thm:extreme-amenability,thm:unit-extreme-amenability}
apply: both groups are exotic, and the additive group is extremely
amenable. By \Cref{cor:bourbaki,cor:no-escape,cor:escape-homomorphisms},
the unit group is Bourbaki bounded, has exact open-ball
width \(\lfloor\varepsilon^{-1}\rfloor+1\) for every \(\varepsilon>0\),
and admits no nonzero escape function or nontrivial continuous
homomorphism to a Hausdorff group with the escape property.
If \(R\) is locally finite, the unit group of
\(\overline A_\alpha(B,R)\) is extremely amenable as well.

\end{document}